\documentclass[11pt]{amsart}

\usepackage{amssymb}

\newtheorem{theorem}{\bf  Theorem}
\newtheorem{lemma}{\bf  Lemma}

\newtheorem{remark}{ \sc Remark}

\begin{document}

\begin{center}
    \textsf{\textbf{\Large Exponential Speedup of the Janashia-Lagvilava Matrix Spectral Factorization Algorithm}} \\[0.5cm]
    \textbf{Ying Wang}${}^1$,
    \textbf{Lasha Ephremidze}${}^{2,*}$,
     \textbf{Ronaldo García Reyes}${}^1$,
     \textbf{Pedro Valdes-Sosa}${}^{1,*}$\\[0.2cm]
    \textit{${}^1$The Clinical Hospital Chengdu Briain Science Institute \\at the University of Electronic Sciences and Technology of China UESTC\\ ${}^2$  Kutaisi International University, Kutaisi, Georgia} \\[0.05cm]
   {\small
    \textit{${}^*$Corresponding authors:} \textsf{lasha.ephremidze@kiu.edu.ge,  pedro.valdes@neuroinformatics-collaboratory.org}} 
\end{center}


\footnotetext{Ying Wang and Lasha Ephremidze are shared first authors}

\bigskip

{\small{\bf Abstract.} Spectral factorization is a powerful mathematical tool with diverse applications in signal processing and beyond. The Janashia-Lagvilava method has emerged as a leading approach for matrix spectral factorization. In this paper, we extend a central equation of the method to the non-commutative case, enabling polynomial coefficients to be represented in block matrix form while preserving the equation's fundamental structure. This generalization results in an exponential speedup for high-dimensional matrices. Our approach addresses challenges in factorizing massive-dimensional matrices encountered in neural data analysis and other practical applications. }

\vskip+0.2cm

\vskip+0.2cm \noindent  {\small {\bf Mathematics Subject Classification:}  47A68 }	

\vskip+0.2cm\noindent  {\small {\em Keywords: } Matrix spectral factorization, Janashia-Lagvilava algorithm, Exponential speedup}

 \bigskip

\section{Introduction}

Spectral factorization is a fundamental mathematical tool with wide-ranging practical applications in neuroscience, engineering and modern technologies. The method was first introduced in the seminal works of Wiener \cite{Wiener49} and Kolmogorov \cite{Kolm41} and has since been extensively applied to solve various problems, many of which are computationally reduced to such factorizations. However, the lack of efficient computational algorithms for matrix spectral factorization (MSF) remained a major bottleneck, rendering many theoretical advances in multidimensional signal processing and systems infeasible. Therefore, since Wiener’s initial efforts \cite{Wie58} to develop a sound computational method for MSF, numerous algorithms have appeared over the decades (see the survey papers \cite{Kuc}, \cite{SayKai} and references therein). In recent years, the Janashia-Lagvilava method \cite{IEEE}, \cite{IEEE-2018} has emerged as a leading MSF algorithm among its competitors (see, e.g., \cite{Robin}). Indeed, many of its advantages, particularly for the factorization of singular matrices, have been demonstrated in \cite{Ed22}. Nevertheless, the method remained incapable of processing very large-scale matrices within a reasonable time frame, posing a significant challenge for modern applications.

In the present paper, we propose a further theoretical development of the Janashia-Lagvilava method, which exponentially reduces the computational time of the corresponding MSF algorithm, making it possible to process massive-dimensional matrices in real time. This significant improvement was achieved by allowing the coefficients in the primary system of equations within the Janashia-Lagvilava method to be matrices, thereby reformulating these equations in a non-commutative framework. Consequently, instead of incrementally making leading principal submatrices analytic in a step-by-step manner—first $2\times2$, then $3\times3$, $4\times4$, and so on—we now process submatrices starting from the main diagonal and expanding in sizes that double sequentially: $2\times2$, $4\times4$, $8\times8$, $16\times16$, etc. This approach significantly enhances computational efficiency for large-scale matrices and extends the applicability of MSF to fields such as neuroscience and machine learning, which typically involve massive datasets.

The paper is organized as follows: Section 2 compiles the notation used throughout the paper. In Section 3, we mathematically formulate the matrix spectral factorization (MSF) problem. Section 4 introduces the main innovations of the improved method. Sections 5 and 6 provide brief descriptions of the old and new versions of the Janashia-Lagvilava MSF algorithm, respectively. Finally, Section 7 presents numerical simulation results demonstrating significant improvements in computational time.

\section{Notation}

Let $I_M$ be the $M\times M$ identity matrix, and $0_M$ be the zero matrix of the same size. For a set $\mathcal{S}$, let $\mathbf{S}^{M\times M}$ be the set of $M\times M$ matrices with entries from $\mathcal{S}$. For $A=[a_{ij}]\in\mathbb{C}^{r\times r}$, let
$$
|A|_\infty=\max\nolimits_{1\leq i,j\leq r}|a_{ij}|.
$$
For a matrix $A$, we denote by $A^T$ its transpose, and by $A^H=\overline{A}^T$ its Hermitian conjugate. However, if a block matrix form of $A$ is specified,
$$
A=\begin{pmatrix}B_{11}&B_{12}&\ldots B_{1m}\\
B_{21}&B_{22}&\ldots B_{2m}\\
\vdots&\vdots&\ldots\vdots\\
B_{m1}&B_{m2}&\ldots B_{mm}\\
           \end{pmatrix},\;\;\;B_{ij}\in\mathbb{C}^{M\times M},
$$
then we define $A^{T_b}$ and $A^{H_b}$ in a block matrix sense:
\begin{equation}\label{S0}
    A^{T_b}=\begin{pmatrix}B_{11}&B_{21}&\ldots B_{m1}\\
B_{12}&B_{22}&\ldots B_{m2}\\
\vdots&\vdots&\ldots\vdots\\
B_{1m}&B_{2m}&\ldots B_{mm}\\
           \end{pmatrix}
\text{ and }
A^{H_b}=\begin{pmatrix}\overline{B_{11}}&\overline{B_{21}}&\ldots \overline{B_{m1}}\\
\overline{B_{12}}&\overline{B_{22}}&\ldots \overline{B_{m2}}\\
\vdots&\vdots&\ldots\vdots\\
\overline{B_{1m}}&\overline{B_{2m}}&\ldots \overline{B_{mm}}\\
           \end{pmatrix}.
\end{equation}

Note that if we take the transposes of each entries of $A^{T_b}$ and $A^{H_b}$, then we will get $A^T$ and $A^H$, respectively, i.e., if $A=[B_{ij}]$ is a block matrix, then
$$
A^T=[B_{ji}^T]\;\;\text{ and }\;\;A^H=[B_{ji}^H].
$$
Let $\mathcal{P}_{\{-n_1,n_2\}}(M)$, where $n_1,n_2\in\mathbb{N}_0=\mathbb{N}\cup\{0\}$ and $M\in\mathbb{N}$, be the set Laurent matrix polynomials of size $M\times M$:
$$
\mathcal{P}_{\{-n_1,n_2\}}(M):=\left\{ \sum\nolimits_{n=-n_1}^{n_2}A_nz^n: \;A_n\in\mathbb{C}^{M\times M}, n=-n_1, -n_1+1,\ldots, n_2\right\}
$$
and suppose $\mathcal{P}^+_N(M):=\mathcal{P}_{\{0,N\}}(M)$ and 
$\mathcal{P}^-_N(M):=\mathcal{P}_{\{-N,0\}}(M)$. Denote also $\mathcal{P}^+(M)=
\cup_{N=0}^\infty \mathcal{P}^+_N(M)$ and $\mathcal{P}^-(M)=
\cup_{N=0}^\infty \mathcal{P}^-_N(M)$.

For $P(z)=\sum_{n=-n_1}^{n_2}A_nz^n$, where $A_n\in\mathbb{C}^{M\times M}$, let $\widetilde{P}$ be its adjoint:
\begin{equation}\label{1.3}
    \widetilde{P}(z)=  \sum\nolimits_{n=-n_1}^{n_2}A_n^Hz^{-n}.
\end{equation}
This operation satisfies the usual conditions: $\widetilde{P_1+P_2}=\widetilde{P_1}+\widetilde{P_2}$; $\widetilde{P_1P_2}=
\widetilde{P_2}\widetilde{P_1}$; $\widetilde{\widetilde{P}}=P$.

Note that for each $z\in\mathbb{C}$ with $|z|=1$, we have $ \widetilde{P}(z)=(P(z))^H$, and if $P(z)=[p_{ij}(z)]$, $p_{ij}\in \mathcal{P}_{\{-n_1,n_2\}}(M)$, has a block matrix-function form, then
$$
 \widetilde{P}(z)= [\widetilde{p_{ji}}(z)].
$$
Obviously $P\in \mathcal{P}^+(M)$ $\Leftrightarrow$ $\widetilde{P}\in \mathcal{P}^-(M)$, and $P\in  \mathcal{P}^+(M)\cap  \mathcal{P}^-(M)$ $\Leftrightarrow$ 
$P\in\mathbb{C}^{M\times M}$.

A matrix function $U\in \mathcal{P}_{\{-n_1,n_2\}}(N)$ is called para-unitary if
$$
U(z) \widetilde{U}(z)=\widetilde{U}(z) U(z)=I_N, \;\text{ for each } z\in\mathbb{C}\setminus\{0\}.
$$
Note that, in this case, $U(z)$ is usual $N\times N$ unitary matrix for each $z$ with $|z|=1$.

The space of $p$-integrable complex valued functions on the unit circle $\mathbb{T}$ is denoted by $L_p(\mathbb{T})$, $p>0$, and 
 $\mathbb{H}_p(\mathbb{D})$ stands for the Hardy space of analytic functions $f$ in the unit disc $\mathbb{D}$ (see, e.g., \cite{Garn}, \cite{Koosis}). In applied sciences, analytic functions are called {\em causal} functions.

 The boundary values function of $f\in \mathbb{H}_p$, i.e.
$
f(e^{i\theta})=\lim\nolimits_{r\to 1}f(re^{i\theta}),
$
exists a.e., belongs to $L_p(\mathbb{T})$, and it uniquely determines the function itself. Therefore, functions from $\mathbb{H}_p$ can be identified with their boundary values and we can assume that $\mathbb{H}_p(\mathbb{D})\subset L_p(\mathbb{T})$. Moreover, there exists a natural characterization of $\mathbb{H}_p$, $p\geq 1$, in terms of the Fourier coefficients of boundary values:
$$
\mathbb{H}_p=\{f\in L_p\,|\, c_n\{f\}=0 \text{ for } n<0\},
\;\text{ where }\;
c_n\{f\}=\frac{1}{2\pi}\int\nolimits_0^{2\pi} f(e^{i\theta}) e^{-in\theta}\,d\theta, \;\;\;n\in\mathbb{Z}.
$$
(Whenever $f\in L_1(\mathbb{T})^{r\times r}$ is a matrix function, we assume that the Forier coefficients are matrices and denote by $C_n\{f\}\in \mathbb{C}^{r\times r}$.)

 A function $f\in \mathbb{H}_p$ is called {\em outer}, if
\begin{equation}\label{7.5}
    |f(0)|=\exp\left(\frac{1}{2\pi}\int\nolimits_0^{2\pi} \log |f(e^{i\theta})|\,d\theta\right)
\end{equation}
and a matrix function $F\in H_1^{r\times r}$ is called outer if its determinant is outer (see \cite{EL10}). 
The right-hand side of \eqref{7.5} is a maximal possible value of $|f(0)|$ in the class of functions $f$ from $\mathbb{H}_p$ with given absolute values on the boundary. In signal processing applications, such functions are sometimes also referred to as  {\em optimal} or {\em minimal phase}.

We use standard MATLAB notation for submatrices of a given matrix 
$A$, say $A(m,1:n)$ or $A(m_1:m_2,n_1:n_2)$, and `ifft' in formula \eqref{ifft} refers to the MATLAB command for the inverse Fourier transform.

\section{Formulation of matrix spectral factorization problem}

Let $S\in L_1(\mathbb{T})^{r\times r}$ be a positive definite (a.e.) matrix function
\begin{equation}\label{aSi}
S(t)=\begin{pmatrix} s_{11}(t)& s_{12}(t)& \cdots&s_{1r}(t)\\
s_{21}(t)& s_{22}(t)& \cdots&s_{2r}(t)\\
\vdots&\vdots&\vdots&\vdots\\s_{r1}(t)& s_{r2}(t)&
\cdots&s_{rr}(t)\end{pmatrix},
\end{equation}
which satisfies  the Paley-Wiener condition 
\begin{equation}\label{9.1}
\int\nolimits_\mathbb{T}\log |\det S(t)|\,dt >-\infty.
\end{equation}
Then, according to the matrix spectral factorization theorem proved by Wiener and Masani \cite{Wie57}, there exists a unique (up to a unitary constant matrix) outer matrix function $S_+\in H_2^{r\times r}$ such that 
\begin{equation}\label{9.2}
S(t)=S_+(t)S_+^H(t).
\end{equation}
The condition \eqref{9.1} is also necessary for the existence of factorization \eqref{9.2}.

In the scalar case, $r=1$, the spectral factor can be found explicitly  (\cite{Garn}, \cite{Koosis})
\begin{equation}\label{9.7}
S_+(z)=\exp\left(\frac{1}{4\pi}\int\nolimits_0^{2\pi} 
\frac{e^{i\theta}+z}{e^{i\theta}-z}
\log |S(e^{i\theta})|\,d\theta\right).
\end{equation}
and, therefore, it is assumed that the spectral factor can be constructed numerically. However, in the matrix case, no such explicit formulas exist, making the determination of a spectral factor $S_+$ for a given matrix function \eqref{aSi} a very demanding problem. 

\section{Introduced innovation}

We generalize the main system of boundary value conditions in Janashia-Lagvilava method (see \cite[Eq. (15)]{IEEE}) for matrix valued functions. Namely, let 
\begin{equation}\label{2.1}
\zeta_j\in \mathcal{P}^-_N(M),\;\; j=1,2,\ldots, m-1,    
\end{equation}
and consider the following system of $m$ conditions
\begin{equation}\label{2.2}
\begin{cases}x_m\cdot \zeta_1-\widetilde{x_1}\in\mathcal{P}^+(M),\\
             x_m \cdot\zeta_2-\widetilde{x_2}\in\mathcal{P}^+(M),\\
             \vdots\\
             x_m \cdot\zeta_{m-1}-\widetilde{x_{m-1}}\in\mathcal{P}^+(M),\\
             \zeta_1\cdot x_1+\zeta_2\cdot x_2+\ldots+\zeta_{m-1}\cdot x_{m-1}
              +\widetilde{x_m}\in\mathcal{P}^+(M),
         \end{cases}
\end{equation}
where $x_i\in \mathcal{P}^+_N(M)$, $i=1,2,\ldots,m$, are unknowns (we emphasize that the sign ``$\cdot$" above indicates matrix production). 

We say that a block vector function
\begin{equation*}
\mathbf{u}(z)=\big(u_1(z),u_2(z),\ldots,u_m(z)\big)^{T_b}, \text{ where } 
u_i\in \mathcal{P}_N^+(M), i=1,2,\ldots,m,
\end{equation*}
is a solution of \eqref{2.2} if and only if all the conditions in \eqref{2.2}
are satisfied by substitution $x_i(z)=u_i(z)$, $i=1,2,\ldots,m$.

\begin{lemma} $($cf. \cite[Lemma 2]{IEEE}$)$
Let \eqref{2.1}  holds and  let
$$
\mathbf{u}(z)=\big(u_1(z),u_2(z),\ldots,u_m(z)\big)^{T_b},\;
u_j\in\mathcal{P}_N^+(M),\, j=1,2,\ldots,m$$ and 
$$
\mathbf{v}(z)=\big(v_1(z),v_2(z),\ldots,v_m(z)\big)^{T_b},\;
v_j\in\mathcal{P}_N^+(M),\,  j=1,2,\ldots,m,
$$
be two $($possibly identical$)$ solutions of the system \eqref{2.2}.
Then
\begin{equation}\label{S5}
\sum_{i=1}^{m-1} \widetilde{u_i}(z) \cdot v_i(z)
+u_m(z)\cdot \widetilde{v_m}(z)=\operatorname{const}\in\mathbb{C}^{M\times M}.
\end{equation}
\end{lemma}

\begin{proof}
Substituting the matrix functions $u_i$ in the first $m-1$ conditions and
the matrix functions $v_i$ in the last condition of \eqref{2.2}, and then
multiplying on the right the first $m-1$ conditions by $v_i$ and the last
condition by $u_m$ on the left,  we get
$$
\begin{cases} u_m\cdot\zeta_1\cdot v_1-\widetilde{u_1}\cdot v_1\in \mathcal{P}^+(M),\\
             u_m\cdot \zeta_2\cdot v_2-\widetilde{u_2}\cdot v_2\in \mathcal{P}^+(M),\\
              \cdot\hskip+1cm \cdot\hskip+1cm \cdot\\
              u_m\cdot\zeta_{m-1}\cdot v_{m-1}-\widetilde{u_{m-1}}\cdot v_{m-1}\in \mathcal{P}^+(M),\\
             u_m\cdot \zeta_1\cdot v_1+u_m\cdot\zeta_2\cdot v_2+\ldots+u_m\cdot\zeta_{m-1}\cdot v_{m-1}
              +u_m\cdot\widetilde{v_m}\in \mathcal{P}^+(M).
         \end{cases}
$$
Subtracting the first $m-1$ conditions from the last condition in
the latter system, we get
\begin{equation}\label{3.3}
\sum_{i=1}^{m-1} \widetilde{u_i}(z) \cdot v_i(z)
+u_m(z)\cdot \widetilde{v_m}(z)\in \mathcal{P}^+(M).
\end{equation}
We can interchange the roles of $u$ and $v$ in the above
discussion to get in a similar manner that
$$
\sum_{i=1}^{m-1} \widetilde{v_i}(z) \cdot u_i(z)
+v_m(z)\cdot \widetilde{u_m}(z)\in \mathcal{P}^+(M),
$$
which in turn implies that 
\begin{equation}\label{3.4}
\sum_{i=1}^{m-1} \widetilde{u_i}(z) \cdot v_i(z)
+u_m(z)\cdot \widetilde{v_m}(z)\in \mathcal{P}^-(M).
\end{equation}
The relations \eqref{3.3} and \eqref{3.4} imply \eqref{S5}.
\end{proof}

The above lemma enables us to provide a constructive proof of the following theorem, which serves as a block matrix generalization of the main computational tool in the Janashia-Lagvilava method.

\begin{theorem}\label{Th1}
Let $M, m$, and $N$ be positive integers. For any $m\times m$ block matrix function $F(z)$ of the  form
\begin{equation}\label{S8}
F(z)=\begin{pmatrix}I_M&0_M&\cdots&0_M&0_M\\
          0_M&I_M&\cdots&0_M&0_M\\
           \vdots&\vdots&\vdots&\vdots&\vdots\\
           0_M&0_M&\cdots&I_M&0_M\\
           \zeta_{1}(z)&\zeta_{2}(z)&\cdots&\zeta_{m-1}(z)&f(z)
           \end{pmatrix},
\end{equation}
where
\begin{equation}\label{S9}
\zeta_j\in \mathcal{P}_N^-(M),\;j=1,2,\ldots, m-1,\text{ and }
f\in\mathcal{P}_N^+(M) \text{ with } \det C_0\{f\}\not=0,
\end{equation}
there exists a para-unitary matrix function $U$ of the form
  \begin{equation}\label{4.2}
U(z)=\begin{pmatrix}u_{11}(z)&u_{12}(z)&\cdots&u_{1m}(z)\\
                           \vdots&\vdots&\vdots&\vdots\\
           u_{m-1,1}(z)&u_{m-1,2}(z)&\cdots&u_{m-1,m}(z)\\[3mm]
           \widetilde{u_{m1}}(z)&\widetilde{u_{m2}}(z)&\cdots&\widetilde{u_{mm}}(z)\\
           \end{pmatrix},
\end{equation}
where
 \begin{equation*}
 u_{ij}(z)\in \mathcal{P}_N^+(M),\;\;i,j=1,2,\ldots,m,
\end{equation*}
with determinant 1, $\det U(z)=1$ for each $z\in\mathbb{C}\setminus\{0\}$,
 such that
  \begin{equation*}
F U\in \big(\mathcal{P}_N^+(M)\big)^{m\times m}.
\end{equation*}
\end{theorem}

\begin{remark}
    In this paper, we apply this theorem specifically for $ m=2 $. However, the general formulation may prove useful for potential modifications of the proposed algorithm.
\end{remark}

\begin{proof}
    Note that $F$ can be formally represented as 
$$
    F=\begin{pmatrix}I_M&0_M&\cdots&0_M&0_M\\
          0_M&I_M&\cdots&0_M&0_M\\
           \vdots&\vdots&\vdots&\vdots&\vdots\\
           0_M&0_M&\cdots&I_M&0_M\\
        0_M&0_M&\cdots&0_M&f
           \end{pmatrix}
    \begin{pmatrix}I_M&0_M&\cdots&0_M&0_M\\
          0_M&I_M&\cdots&0_M&0_M\\
           \vdots&\vdots&\vdots&\vdots&\vdots\\
           0_M&0_M&\cdots&I_M&0_M\\
           f^{-1}\zeta_{1}&f^{-1}\zeta_{2}&\cdots&f^{-1}\zeta_{m-1}&I_M
           \end{pmatrix},
$$
where, under $f^{-1}$ we understand a formal power series with matrix coefficients $f^{-1}(z)=\sum_{n=0}^\infty A_nz^n$ such that $f^{-1}(z)f(z)=I_M$. Furthermore, each $f^{-1}\zeta_j$, $j=1,2,\ldots,m-1$, can be decomposed as a sum  $f^{-1}\zeta_j=(f^{-1}\zeta_j-\zeta_j^-)+\zeta_j^-$, where $\zeta_j^-\in \mathcal{P}_N^-(M)$ and $C_n\{f^{-1}\zeta_j-\zeta_j^-\}=0$ for every $n<0$. Thus, if we introduce the notation $f^{-1}\zeta_j-\zeta_j^-=:\zeta_j^+$, we have the following decomposition of $F$:
$$
F=\begin{pmatrix}I_M&\cdots&0_M&0_M\\
           \vdots&\vdots&\vdots&\vdots\\
           0_M&\cdots&I_M&0_M\\
        0_M&\cdots&0_M&f
           \end{pmatrix}
   \begin{pmatrix}I_M&\cdots&0_M&0_M\\
             \vdots&\vdots&\vdots&\vdots\\
           0_M&\cdots&I_M&0_M\\
           \zeta_{1}^+&\cdots&\zeta_{m-1}^+&I_M
           \end{pmatrix}
            \begin{pmatrix}I_M&\cdots&0_M&0_M\\
             \vdots&\vdots&\vdots&\vdots\\
           0_M&\cdots&I_M&0_M\\
           \zeta_{1}^-&\cdots&\zeta_{m-1}^-&I_M
           \end{pmatrix},
$$
where the negative-indexed Fourier coefficients of each entry from the first two matrices in the above product vanish. Hence, we can assume throughout the proof below, without loss of generality, that $f(z)=I_M$ in \eqref{S8}.

For a given block matrix functions $\zeta_j$ in \eqref{S9}, we construct $m$ independent solutions of the system \eqref{2.2}, which will be $m$ column blocks of the matrix \eqref{4.2}. Let 
$$
\zeta_j(t)=
 \sum_{n=0}^N\gamma_{jn}z^{-n}, \text{ where } \gamma_{jn}\in \mathbb{C}^{M\times M}, \;n=0,1,\ldots,N;\,j=1,2,\ldots,m-1,
$$
and consider $(N+1)\times(N+1)$ block-Hankel matrices
$$
\Gamma_j=\begin{pmatrix}\gamma_{j0}&\gamma_{j1}&\gamma_{j2}&\cdots&\gamma_{j,N-1}&\gamma_{jN}\\
        \gamma_{j1}&\gamma_{j2}&\gamma_{j3}&\cdots&\gamma_{jN}&0_M\\
        \gamma_{j2}&\gamma_{j3}&\gamma_{j4}&\cdots&0_M&0_M\\
        \cdot&\cdot&\cdot&\cdots&\cdot&\cdot\\
        \gamma_{jN}&0_M&0_M&\cdots&0_M&0_M\end{pmatrix},\; j=1,2,\ldots, m-1.
$$
We search solutions of the system \eqref{2.2} in the form 
$$
x_i^+(z)=\sum_{n=0}^N a_{in} z^n, \text{ where } a_{in}\in\mathbb{C}^{M\times M},\;\;n=0,1,\ldots,N; \, i=1,2,\ldots, m,
$$
and introduce the following block rows
$$
X_i=(a_{i0},a_{i1},\ldots,a_{iN})\in\mathbb{C}^{M\times M(N+1)},\;\;i=1,2,\ldots,m,
$$
Note that (see \eqref{S0})
$$
X_i^H=(a^{H}_{i0},a^{H}_{i1},\ldots,a^{H}_{iN})^{T_b} \text{ and }
\big(X_i^{H_b}\big)^{T_b}=(a^{H}_{i0},a^{H}_{i1},\ldots,a^{H}_{iN})=\big(X_i^{T_b}\big)^H.
$$
We equate the block matrix coefficients of the negative powers of polynomials in \eqref{2.2} to zero, resulting in the following system of linear algebraic equations with block matrix coefficients:
$$
\begin{cases}
    X_m\cdot \Gamma_1-\big(X_1^{T_b}\big)^H=\mathbb{O}\\
    \vdots\\
    X_m\cdot \Gamma_i-\big(X_i^{T_b}\big)^H=\mathbb{I}\\
    \vdots\\
    X_m\cdot \Gamma_{m-1}-\big(X_{m-1}^{T_b}\big)^H=\mathbb{O}\\
    \Gamma_1\cdot X_1^{T_b}+\Gamma_2\cdot X_2^{T_b}+\ldots+
    \Gamma_{m-1}\cdot X_{m-1}^{T_b}+X_m^H=\mathbb{O}^{T_b},    
\end{cases}
$$
where 
$$
\mathbb{O}=(0_M,0_M,\ldots,0_M)\in\mathbf{C}^{M\times M(N+1)}\;\text{ and }
\mathbb{I}=(I_M,0_M,\ldots,0_M)\in\mathbf{C}^{M\times M(N+1)}.
$$
If we take Hermitian conjugate of the first $m-1$ equations in the above system, we get
$$
\begin{cases}
     \Gamma_1 \cdot  X_m^H-X_1^{T_b}=\mathbb{O}^{T_b}\\
    \vdots\\
    \Gamma_i \cdot  X_m^H-X_i^{T_b}=\mathbb{I}^{T_b}\\
    \vdots\\
    \Gamma_{m-1} \cdot  X_m^H-X_{m-1}^{T_b}=\mathbb{O}^{T_b}\\
    \Gamma_1\cdot X_1^{T_b}+\Gamma_2\cdot X_2^{T_b}+\ldots+
    \Gamma_{m-1}\cdot X_{m-1}^{T_b}+X_m^H=\mathbb{O}^{T_b}.    
\end{cases}
$$
Next, we can directly follow the approach outlined in the proof of Theorem 1 in [1], disregarding the distinction between block matrices and numerical coefficients. Therefore, we do not repeat the proof here.
\end{proof}

\smallskip

We now proceed to compare the old and new matrix spectral factorization algorithms, with input $S$, a positive-definite $ r \times r $ matrix function (on the unit circle), and output $ S_+ $, its approximate spectral factor.

\section{ A general description of the old algorithm proposed in \cite{IEEE}, \cite{IEEE-2018}}

{\bf Procedure 1.} Lower-upper triangular factorization of $S$,
$$
S=Q\cdot Q^H,
$$
with corresponding scalar spectral factors on the main diagonal of $Q$. This can be achieved by applying Cholesky factorization pointwise at each frequency resolution node, followed by performing scalar spectral factorization of the diagonal entries according to formula \eqref{9.7}.

{\bf Procedure 2.} We represent $S_+$ as a product
$$
S_+=Q\mathbf{U}_2 \mathbf{U}_3\cdots \mathbf{U}_r.
$$
To this end, we perform the following recurrent steps, which make the leading $ m \times m $ principal submatrices of $ Q $ analytic in turn, for $ m = 2, 3, \ldots, r $.

{\bf Step 1.} Take the first $m$ (nonzero) entries in the $m$th row of the 
$$
Q_{m-1}:=Q\mathbf{U}_2 \mathbf{U}_3\cdots \mathbf{U}_{m-1}
$$
(which is assumed to be constructed in the previous recurrent step), namely,
$$
(\zeta_1, \zeta_2,\ldots,\zeta_{m-1}, f_m),
$$
and determine large $N$ for which all the Fourier coefficients $c_n\{\zeta_j\}$, $n<-N$, $j=1,2,\ldots,m-1$, are very small with absolute value.

{\bf Step 2.} Create a $m\times m$ matrix function
$$
F=\begin{pmatrix}1&0&\cdots&0&0\\
          0&1&\cdots&0&0\\
           \vdots&\vdots&\vdots&\vdots&\vdots\\
           0&0&\cdots&1&0\\
           \zeta^{\{N\}}_{1}&\zeta^{\{N\}}_{2}&\cdots&\zeta^{\{N\}}_{m-1}&f_m
           \end{pmatrix},\;\text{ where } \zeta^{\{N\}}_j(z)=\sum_{n=-N}^{-1}c_n\{\zeta_j\}z^n,
$$
and, using \cite[Theorem 2]{IEEE}, construct a unitary (on the unit circle) matrix function $U$ of special structure \eqref{4.2}, where $u_{ij}\in \mathcal{P}^+_N=\mathcal{P}^+_N(1)$, with determinant 1, such that 
$$
FU\in \big(\mathcal{P}^+\big)^{m\times m}.
$$
{\bf Step 3.} Multiply the first $m$ columns of $Q_{m-1}$ by $U$, which produces $Q_m$.

\smallskip

For $m=r$, we have $S_+=Q_r$.

\smallskip

\section{ A general description of the new algorithm}

Without loss of generality, we assume that $r=2^p$. If this is not the case, we can add ones to the diagonal of $S$ thereby artificially extending its dimension. The algorithm can be suitably modified if $r$, or a slightly larger number than 
$r$, follows a different simple factorization pattern, by employing an alternative strategy for block creation and parallelization than the one presented below.

The first step in the new algorithm is identical to that in the old one. 

{\bf Procedure 2.}  We represent $S_+$ as a product
$$
S_+=Q\mathbf{U}_2 \mathbf{U}_4\cdots \mathbf{U}_{2^m}\mathbf{U}_r.
$$
To this end, we perform the following recurrent steps, which successively make the block diagonal entries of size $2^m \times 2^m$ analytic for $m = 2, 3, \ldots, p$.

Assume that
$$
Q_{m-1}:=Q\mathbf{U}_2 \mathbf{U}_4\cdots \mathbf{U}_{2^{m-1}}
$$
is already constructed, which has the following properties:

(i) block diagonal entries $Q_{m-1}[2^m(k-1)+1:2^mk\, , \, 2^m(k-1)+1:2^mk]$, for $k=1,2,\ldots 2^{p-m}$, are analytic; 

(ii) upper off-block diagonal entries are zero;

(iii) lower off-block diagonal entries can be non-analytic;

and from $k=1$ to $k=2^{p-m-1}$ perform the following steps, which can be executed in parallel:

{\bf Step 1.} Label the non-analytic $2^m\times 2^m$ matrix function 
$$Q_{m-1}[2^m(2k-1)+1:2^{m+1}k\, , \, 2^m(2k-2)+1:2^m(2k-1)]$$
 and the block diagonal entry 
$$Q_{m-1}[2^m(2k-1)+1:2^{m+1}k\, , \, 2^m(2k-1)+1:2^{m+1}k]$$ 
 as $\zeta$ and $f$, respectively. Determine a sufficiently large 
$N$ such that all block Fourier coefficients 
$C_n\{\zeta\}$, for $n<-N$ have entries with very small absolute values.

{\bf Step 2.} Create a block matrix function
$$
F=\begin{pmatrix}
    I_{2^m} &  0_{2^m}\\
    \zeta_N & f_N
\end{pmatrix},
$$
where
$$
\zeta_N(z)=\sum_{n=-N}^{-1} C_n\{\zeta\} z^n\in\mathcal{P}^-_N(2^m)
\;\;\text{ and } \;\;
f_N(z)=\sum_{n=0}^{N} C_n\{f\} z^n\in\mathcal{P}^+_N(2^m),
$$
and, using Theorem \ref{Th1}, construct a para-unitary matrix $U$ of a special block matrix structure 
$$
U(z)=\begin{pmatrix}
    u_{11}(z)& u_{12}(z)\\
    \widetilde{u_{21}}(z)& \widetilde{u_{22}}(z)
\end{pmatrix},  \;\;\text{ where } u_{ij}\in \mathcal{P}^+_N(2^m),
$$
with $\det U(z)=1$, such that 
$$
F\cdot U\in \big(\mathcal{P}^+_N(2^m)\big)^{2\times 2}.
$$

{\bf Step 3}. Multiply nonzero entries of the two block columns of $Q_{m-1}$ containing $\zeta$ and $f$, namely $Q_{m-1}[2^m(2k-2)+1:2^p\, , \, 2^m(2k-2)+1:2^{m+1}k]$ by $U$.

After performing these steps for all $k$, the next block matrix $Q_m$ is ready.

For $m=p$, we have $S_+=Q_p$.

\section{Numerical simulations}

To compare the new and old algorithms, we investigated the performance of the corresponding MATLAB codes on the same dataset. A random positive definite $1024\times 1024$ matrix function was constructed as follows: we took
$$
S_n(t)=\sum_{k=0}^nA_nt^n,
$$
where $n=10$, and $A_k\in\mathbb{C}^{1024\times 1024}$, $k=0,1,\ldots,10$, are matrix coefficients with entries selected randomly from the uniform distribution on $[-1,1]$, and let
\begin{equation}\label{Sgam}
S(t)=\sum\nolimits_{k=0}^nA_nt^n \sum\nolimits_{k=0}^nA_n^Ht^{-n}.
\end{equation}
Then, the values of the matrix function  $S$ were computed at the points $z_j=\exp(2\pi ij/N)$, $j=0,1,\ldots,N-1$, where $N=2^9=512$; that is, the data was artificially generated with a frequency resolution of $2^p$  points, where $p=9$. All computations for both the old and new algorithms, as described above, were performed at these resolutions, and the spectral factor 
$S_+$ was obtained for the same points $z_j=\exp(2\pi ij/N)$, $j=0,1,\ldots,N-1$. Since the correct values of 
$S_+$ are unknown for the random matrix in equation \eqref{Sgam}, we require criteria to estimate the accuracy of the results. To this end, we introduce
$$
C_1=\max\nolimits_{j=0,1,\ldots,N-1}|S(z_j)-S_+(z_j)S_+^H(z_j)|_\infty
$$
and
\begin{equation}\label{ifft}
C_2=\max\nolimits_{j=N/2+1,\ldots,N-1}|{\rm ifft}(S_+)(z_j)|_\infty
\end{equation}
The closeness of $C_2$ to zero indicates that $S_+$ is `approximately causal'. The structure of the algorithm guarantees that $\det(S_+(z))\not=0$ for $|z|<1$ within the accuracy of the round-off errors in the scalar spectral factorization, so we do not test the result based on this criterion. 

For both algorithms, we obtained nearly identical values for 
$C_1$ and $C_2$:
$$
C_1<10^{-11}\;\;\;\text{ and }\;\;\;C_2<10^{-2}
$$
However, the running times for the old and new algorithms, $T_{old}$ and $T_{new}$, were significantly different:
$$
T_{old}\approx 2.5 \text{ hours} \;\;\;\text{ and }
T_{new}\approx 5 \text{min}.
$$
We emphasize that this occurred without leveraging the opportunity to parallelize Steps 1, 2, and 3 in Procedure 2 of the new algorithm.

\subsection*{Acknoledgmets}
Lasha Ephremidze is grateful to the University of Electronic Science and Technology of China for excellent working conditions and hospitality during his stay to Chengdu.

\subsection*{Author Contribution.} This work began with a working visit by LE to the University of Electronic Science and Technology, driven by his decades of expertise in the algorithm, which proved highly fruitful and enabled intense brainstorming that significantly shaped the research. The authors contributed collaboratively to the project: YW conceived the idea of parallelization, optimized its implementation, conducted numerical simulations to test various approaches, and developed the final MATLAB code; LE provided an in-depth understanding of the previous method and its generalizations; RGR offered valuable insights into the existing method and introduced innovative ideas for incorporating randomization; and PAVS initiated the project, provided overall guidance, and motivated the team to pursue algorithmic enhancements. YW and LE are co-first authors, contributing equally to this work, while LE and PAVS served as co-corresponding authors.


\def\cprime{$'$}
\providecommand{\bysame}{\leavevmode\hbox to3em{\hrulefill}\thinspace}
\providecommand{\MR}{\relax\ifhmode\unskip\space\fi MR }
\providecommand{\MRhref}[2]{%
  \href{http://www.ams.org/mathscinet-getitem?mr=#1}{#2}
}
\providecommand{\href}[2]{#2}

\end{document}